\documentclass[11pt]{amsart}

\usepackage[T1]{fontenc}
\usepackage{lmodern}
\usepackage[protrusion=true,expansion=false]{microtype}
\usepackage[margin=1in]{geometry}
\usepackage{booktabs}
\usepackage{float}
\usepackage[hidelinks]{hyperref}
\hypersetup{
  pdftitle={The Asayama--Matsumoto conjecture and a refined discrepancy bound},
  pdfauthor={Alessio Basti and Tommaso Cremaschi},
  pdfsubject={Polychromatic 2-colorings of plane triangulations},
  pdfkeywords={plane triangulation, polychromatic coloring, discrepancy, graph coloring}
}

\newtheorem{theorem}{Theorem}[section]

\newcommand{\disc}{\operatorname{disc}}
\newcommand{\ceil}[1]{\left\lceil #1\right\rceil}
\newcommand{\floor}[1]{\left\lfloor #1\right\rfloor}

\title[The Asayama--Matsumoto conjecture]
{The Asayama--Matsumoto conjecture\\
and a refined discrepancy bound}

\author{Alessio Basti}
\address{Department of Engineering and Geology, ``G. d'Annunzio'' University of Chieti--Pescara, Italy}

\author{Tommaso Cremaschi}
\address{School of Mathematics, Trinity College Dublin, Ireland}

\date{August 21, 2026}

\begin{document}
\raggedbottom
\emergencystretch=1em

\begin{abstract}
Every $n$-vertex plane triangulation admits a polychromatic red--blue
vertex coloring with discrepancy at most
$n-2\ceil{n/3}\le\floor{n/3}$, resolving a conjecture of Asayama and
Matsumoto. The proof follows by combining \cite{LoyolaEtAl2026,KawarabayashiEtAl2026}. For $n\ge6$ and $n\not\equiv5\pmod6$, we prove the sharper
bound $n-2\ceil{(n+2)/3}$.  An exhaustive census of all $9{,}150$
non-isomorphic sphere triangulations with $4\le n\le12$ verifies the bounds
and also confirms the sharper estimate at $n=11$, the first order in the
remaining open residue class.
\end{abstract}

\maketitle

\section{Introduction}

A \emph{plane triangulation} is a simple planar graph embedded so that every
face, including the outer face, is a triangle. A red--blue vertex coloring is
\emph{polychromatic} if every facial triangle contains both colors. If the
two color classes are \(R\) and \(B\), their discrepancy is
\(\bigl||R|-|B|\bigr|\); we denote by \(\disc(T)\) the minimum discrepancy
over all polychromatic colorings of \(T\).

Asayama and Matsumoto~\cite{AsayamaMatsumoto2022} proved that every
\(n\)-vertex triangulation admits a polychromatic coloring with discrepancy
at most \((5n-16)/9\). They also constructed an infinite family with
discrepancy at least \(n/3-2\), showing that the leading constant in any
universal upper bound cannot be smaller than \(1/3\). This led them to
conjecture that
\begin{equation}\label{eq:AM-conjecture}
  \disc(T)\le \frac{n}{3}.
\end{equation}
Arevalo Loyola, Biniaz, Bose, and Shermer~\cite{LoyolaEtAl2026} 
improved the general upper bound to \((3n-16)/7\).

Our starting point is the balanced four-color theorem of Kawarabayashi,
Yoneda, and Yoneda~\cite{KawarabayashiEtAl2026}: every planar graph has a
proper $4$-coloring in which no color class contains half of the vertices.
A facial triangle uses three distinct proper colors.  Therefore, after
partitioning the four proper colors into two pairs, every face meets both
pairs.  Choosing the pairing carefully yields
\begin{equation}\label{eq:baseline-intro}
  \disc(T)\le B(n):=n-2\ceil{\frac n3}
  \le\floor{\frac n3},
\end{equation}
which proves~\eqref{eq:AM-conjecture}. We then sharpen the integer analysis.  For every $n\ge6$ with
$n\not\equiv5\pmod6$, we prove
\begin{equation}\label{eq:refined-intro}
  \disc(T)\le U(n):=n-2\ceil{\frac{n+2}{3}}.
\end{equation}
The new expression improves~\eqref{eq:baseline-intro} by two when
$n\equiv0$ or $2\pmod3$, and coincides with it when $n\equiv1\pmod3$.
No claim beyond the universal bound is made for $n\equiv5\pmod6$; whether
\eqref{eq:refined-intro} also holds in that residue class remains open.

Finally, we test the estimates by exact computation.  We enumerate every
non-isomorphic sphere triangulation with $4\le n\le12$ and compute its exact
minimum discrepancy by checking all red--blue colorings up to global color
exchange.  This includes all $1{,}249$ triangulations on $11$ vertices, the
first order in the open residue class.  Apart from the balanced four-color
theorem, the proofs below are elementary and are given in full.

\section{The bounds}\label{sec:universal}
In this section we prove our two bounds.

\begin{theorem}[Asayama--Matsumoto conjecture]\label{thm:universal}
Every plane triangulation $T$ on $n\ge3$ vertices satisfies
\begin{equation}\label{eq:universal}
  \disc(T)\le n-2\ceil{\frac n3}\le\floor{\frac n3}.
\end{equation}
Moreover, a coloring whose discrepancy satisfies the first bound can be found in
$O(n\log n)$ time.
\end{theorem}

\begin{proof}
By Corollary~17 of Kawarabayashi, Yoneda, and
Yoneda~\cite{KawarabayashiEtAl2026}, the underlying planar graph of $T$
has a proper $4$-coloring whose color classes $V_1,V_2,V_3,V_4$ may be
indexed so that, with $n_i=|V_i|$,
\[
  n_1\ge n_2\ge n_3\ge n_4
  \quad\text{and}\quad
  n_1\le \ceil{\frac{n-2}{2}}
       =\floor{\frac{n-1}{2}}
       <\frac n2.
\]
Define
\[
  R=V_1\cup V_4,
  \qquad
  B=V_2\cup V_3,
\]
and put $r=|R|=n_1+n_4$.  Every facial triangle has three pairwise
adjacent vertices and therefore uses three distinct proper colors.  Each
of $R$ and $B$ comprises only two proper colors, so every face meets both
sets.  Thus the merged coloring is polychromatic.  This is the $2+2$
merging argument of \cite[Lemma~14]{LoyolaEtAl2026}.

It remains to check the balance.  Since $r\ge n_2,n_3$,
\[
  n=r+n_2+n_3\le3r,
  \qquad\text{so}\qquad
  r\ge\frac n3.
\]
Because $n_4$ is the smallest of $n_2,n_3,n_4$,
\[
  3n_4\le n_2+n_3+n_4=n-n_1.
\]
Consequently,
\[
  r=n_1+n_4
    \le n_1+\frac{n-n_1}{3}
    =\frac{n+2n_1}{3}
    <\frac{2n}{3}.
\]
Hence $n/3\le r<2n/3$.  Since both $r$ and $n-r$ are integers,
\[
  \min\{|R|,|B|\}
  =\min\{r,n-r\}
  \ge \ceil{\frac n3}.
\]
Therefore
\[
  \bigl||R|-|B|\bigr|
  =n-2\min\{|R|,|B|\}
  \le n-2\ceil{\frac n3}.
\]
For $n=3q,3q+1,3q+2$, the last expression is respectively
$q,q-1,q$, and is therefore at most $\floor{n/3}$.

The balanced proper $4$-coloring is computable in $O(n\log n)$ time by
\cite[Corollary~17]{KawarabayashiEtAl2026}; counting and merging its four
classes take linear additional time.
\end{proof}

For comparison, the two integer bounds are shown below.

\begin{table}[H]
\centering
\small
\begin{tabular}{c|cccccc}
\toprule
$n$ & $6k$ & $6k+1$ & $6k+2$ & $6k+3$ & $6k+4$ & $6k+5$ \\
\midrule
$B(n)$ & $2k$ & $2k-1$ & $2k$ & $2k+1$ & $2k$ & $2k+1$ \\
$U(n)$ & $2k-2$ & $2k-1$ & $2k-2$ & $2k-1$ & $2k$ & $2k-1$ \\
\bottomrule
\end{tabular}
\caption{The universal bound $B(n)$ and the refined target $U(n)$.  The last
column is not covered by Theorem~\ref{thm:refined}.}
\label{tab:bounds}
\end{table}

\begin{theorem}[Refined bound]\label{thm:refined}
Let $T$ be a plane triangulation on $n\ge6$ vertices.  If
$n\not\equiv5\pmod6$, then
\begin{equation}\label{eq:refined}
  \disc(T)\le n-2\ceil{\frac{n+2}{3}}.
\end{equation}
Moreover, a coloring satisfying this estimate can be found in $O(n\log n)$ time.
\end{theorem}

\begin{proof}
Take a proper $4$-coloring supplied by
\cite[Corollary~17]{KawarabayashiEtAl2026}, and order its color classes
$V_1,V_2,V_3,V_4$ so that, writing
\[
  a=|V_1|\ge b=|V_2|\ge c=|V_3|\ge d=|V_4|,
\]
we have
\begin{equation}\label{eq:balanced-four-refined}
  a\le\ceil{\frac{n-2}{2}}<\frac n2.
\end{equation}
Set
\[
  R=V_1\cup V_4,
  \qquad
  B=V_2\cup V_3,
  \qquad
  r=a+d,
  \qquad
  s=b+c=n-r.
\]
As in the proof of Theorem~\ref{thm:universal}, this merge is
polychromatic.  Moreover, $r\ge b,c$ gives $r\ge n/3$, while
$d\le b,c$ and~\eqref{eq:balanced-four-refined} give
\[
  r=a+d\le\frac{n+2a}{3}<\frac{2n}{3}.
\]
Thus
\begin{equation}\label{eq:baseline-both}
  r,s\ge\ceil{\frac n3}.
\end{equation}
We show that the smaller class in fact has size at least
$\ceil{(n+2)/3}$.

Suppose first that $n=3q$.  By~\eqref{eq:baseline-both}, $r\ge q$,
and $r<2q$ implies $s\ge q+1$.  If $r\ge q+1$, we are done.  If
$r=q$, then $b,c\le r$ and $b+c=2q$, so $b=c=q$.  Since $a\ge b$
and $a+d=q$, we obtain $a=q$ and $d=0$.  Every face therefore contains
one vertex from each of $V_1,V_2,V_3$.  Color $V_2$ red, $V_3$ blue,
one vertex of $V_1$ red, and the remaining vertices of $V_1$ blue.  The
coloring is polychromatic and its two classes have sizes $q+1$ and
$2q-1$, both at least $q+1$ because $q\ge2$.  This direct split is a special case of the general result of Asayama and
Matsumoto that every properly $3$-colorable triangulation has a balanced
polychromatic $2$-coloring~\cite{AsayamaMatsumoto2022}.

If $n=3q+1$, then
\[
  \ceil{\frac n3}=q+1=\ceil{\frac{n+2}{3}},
\]
so Theorem~\ref{thm:universal} already gives the claim.

It remains to consider $n=3q+2$.  The assumption
$n\not\equiv5\pmod6$ says that $q$ is even, and $q\ge2$ because
$n\ge6$.  By~\eqref{eq:baseline-both}, $r,s\ge q+1$.

If $r=q+1$, then $b+c=2q+1$.  Since $b,c\le r$ and $b\ge c$, we get
$b=q+1$ and $c=q$.  The inequalities $a\ge b$ and $a+d=q+1$ then give
$a=q+1$ and $d=0$.  Again every face contains one vertex from each of
$V_1,V_2,V_3$.  Color $V_2$ red, $V_3$ blue, one vertex of $V_1$ red,
and the remaining vertices of $V_1$ blue.  The two classes have sizes
$q+2$ and $2q$, both at least $q+2$.

Finally, suppose that $s=q+1$, so $r=2q+1$.  Since $q$ is even,
\eqref{eq:balanced-four-refined} gives $a\le3q/2$, and hence
\[
  d=r-a\ge\frac q2+1.
\]
As $b,c\ge d$, it follows that
\[
  s=b+c\ge2d\ge q+2,
\]
a contradiction.  Thus neither merged class has size $q+1$, and both
have size at least
$q+2=\ceil{(n+2)/3}$.  This proves~\eqref{eq:refined}.

The construction consists of the $O(n\log n)$ balanced four-coloring
algorithm, class-size comparisons, the merge above, and, in the two
three-color cases, the choice of a single vertex.  Its total running time
is therefore $O(n\log n)$.
\end{proof}

\section{Exact computational verification}\label{sec:computation}

We performed an exhaustive census of all non-isomorphic sphere triangulations
with $4\le n\le12$.  The experiment is exact rather than random: every
triangulation in this range and every red--blue coloring, up to exchanging the
two colors globally, is checked.

\subsection{Method}

For each $n$, the program starts from one sphere triangulation and traverses
the graph of legal diagonal flips.  Wagner's theorem~\cite{Wagner1936}
ensures that this traversal reaches every triangulation of the same order.
Exact graph-isomorphism tests retain one representative of each isomorphism
class.  Since simple triangulations with at least four vertices are
$3$-connected, Whitney's theorem~\cite{Whitney1932} makes their sphere
embeddings unique up to reflection.  The resulting numbers of representatives
are
\[
  1,1,2,5,14,50,233,1249,7595
  \qquad (n=4,5,\ldots,12),
\]
in agreement with the standard census produced by \textsc{plantri}
\cite{BrinkmannMcKay2007}.

For each representative, fix one vertex $v_0$ and enumerate the red sets
$R\subseteq V(T)$ with $v_0\in R$.  This leaves one coloring from each
complementary pair and reduces the search to $2^{n-1}$ masks.  A mask is
accepted precisely when every facial triangle $f$ satisfies
\[
  0<|R\cap f|<3.
\]
Among the accepted masks, the program minimizes $|2|R|-n|$.  The accompanying
script \texttt{discrepancy\_census.py} implements the generation and the
optimization.

\subsection{Results}

\begin{table}[H]
\centering
\small
\setlength{\tabcolsep}{5pt}
\begin{tabular}{r r r r r r r r}
\toprule
$n$ & number & $\disc=0$ & $\disc=1$ & $\disc=2$
& maximum & $B(n)$ & $U(n)$ \\
\midrule
4  & 1    & 1    & 0    & 0  & 0 & 0 & -- \\
5  & 1    & 0    & 1    & 0  & 1 & 1 & -- \\
6  & 2    & 2    & 0    & 0  & 0 & 2 & 0 \\
7  & 5    & 0    & 5    & 0  & 1 & 1 & 1 \\
8  & 14   & 14   & 0    & 0  & 0 & 2 & 0 \\
9  & 50   & 0    & 50   & 0  & 1 & 3 & 1 \\
10 & 233  & 231  & 0    & 2  & 2 & 2 & 2 \\
11 & 1249 & 0    & 1249 & 0  & 1 & 3 & $1^{\ast}$ \\
12 & 7595 & 7579 & 0    & 16 & 2 & 4 & 2 \\
\bottomrule
\end{tabular}
\caption{Exact discrepancy distribution for all $9{,}150$ non-isomorphic
sphere triangulations with $4\le n\le12$.  At $n=11$, the starred value is a
computationally verified target, not a consequence of Theorem~\ref{thm:refined}.}
\label{tab:census}
\end{table}

Every triangulation in the census satisfies the universal bound, and every
order covered by Theorem~\ref{thm:refined} satisfies the refined bound.  The
refined bound is attained: exactly two triangulations on $10$ vertices and
sixteen on $12$ vertices have discrepancy $2$.  Overall, $9{,}132$ of the
$9{,}150$ triangulations attain the parity minimum, while the remaining
$18$ have discrepancy $2$.

At the first open order, $n=11$, all $1{,}249$ triangulations have discrepancy
$1=U(11)$.  This exhaustive finite verification is evidence for the same
estimate in the residue class $n\equiv5\pmod6$, but it does not prove it for
arbitrary $n$.  No general result beyond the universal bound of
Theorem~\ref{thm:universal} is claimed for that residue class.

\end{document}